\documentclass[11pt,oneside,reqno]{amsart}
\usepackage{amsmath,amssymb,amsthm}
\usepackage{mathrsfs} % script fonts
\usepackage[abbrev,nobysame]{amsrefs}% AMS refs
\usepackage[hidelinks]{hyperref}

\newtheorem{thm}{Theorem}%[section]
\newtheorem*{thm*}{Theorem}

\newtheorem*{lem*}{Lemma}
\newtheorem{prop}{Proposition}
\newtheorem*{prop*}{Proposition}

\newtheorem*{cor*}{Corollary}

\newtheorem*{conj*}{Conjecture}

\theoremstyle{definition}

\newtheorem*{defn*}{Definition}
\newtheorem{example}{Example}
\newtheorem*{example*}{Example}

\newtheorem*{prob*}{Problem}
\newtheorem{rmk}{Remark}
\newtheorem*{rmk*}{Remark}
\newtheorem{quest}{Question}
\newtheorem*{quest*}{Question}

\renewcommand{\labelenumi}{(\alph{enumi})}
\newcommand{\R}{\mathbb{R}} % real numbers
\newcommand{\Z}{\mathbb{Z}} % integers
\newcommand{\A}{\mathcal{A}} % cal A
\newcommand{\End}{\operatorname{End}} % endomorphisms
 
\newcommand{\Sym}{\mathfrak{S}} % symmetric group
\newcommand{\Ptn}{\mathcal{P}} % partition algebra
\newcommand{\V}{\mathbf{V}} % vector space
\newcommand{\vv}{\mathbf{v}} % basis elt
\newcommand{\im}{\operatorname{im}} % image
\newcommand{\nn}{[n]} % interval \Z[1, ..., n]
\newcommand{\CC}{\operatorname{Cons}} % consecutive cycles
\newcommand{\ov}{\overline}
\newcommand{\ul}{\underline}
\newcommand{\Hec}{\mathcal{H}} % Hecke algebra
\newcommand{\LP}{\Z[v,v^{-1}]} % Laurent poly ring
\newcommand{\IS}{\operatorname{IS}} % len of max incr subseq
\newcommand{\DS}{\operatorname{DS}} % len of max decr subseq
\newcommand{\rev}{\overleftarrow} % reverse arrow

\newcommand{\RSK}{\operatorname{RSK}} % RSK-shape 
\newcommand{\cell}{\mathfrak{R}} % KL cell
\newcommand{\fs}{\mathfrak{s}} % tab s
\newcommand{\ft}{\mathfrak{t}} % tab t
\newcommand{\Tab}{\mathbb{T}} % tableaux set
\newcommand{\dom}{\unrhd}  %dominates
\newcommand{\notdom}{\ntrianglerighteq} %doesn't dominate
\newcommand{\lessdom}{\unlhd}  %less dominant
\newcommand{\slessdom}{\lhd}  %strictly less dominant
\newcommand{\notlessdom}{\ntrianglelefteq}  %not less dominant
\title[Doubly stochastic matrices and Schur--Weyl duality]%
      {Doubly stochastic matrices and Schur--Weyl duality
        for partition algebras}
      \author{Stephen R.~Doty}
      \email{doty@math.luc.edu}
      \address{Department of Mathematics and Statistics,
        Loyola University Chicago, Chicago, IL 60660 USA}
      
\begin{document}
\begin{abstract}\noindent
We prove that the permutations of $\{1,\dots, n\}$ having an
increasing (resp., decreasing) subsequence of length $n-r$ index a
subset of the set of all $r$th Kronecker powers of $n \times n$
permutation matrices which is a basis for the linear span of that
set. Thanks to a known Schur--Weyl duality, this gives a new basis for
the centralizer algebra of the partition algebra acting on the $r$th
tensor power of a vector space. We give some related results on the
set of doubly stochastic matrices in that algebra.
\end{abstract}\maketitle

\section*{Introduction}\noindent
Let $\V$ be a free $\Bbbk$-module with basis $\{\vv_1, \dots, \vv_n\}$
over a unital commutative ring $\Bbbk$; identify $m \in \Z$ with its
image under the natural ring morphism $\Z \to \Bbbk$. The symmetric
group $W_n$ on $n$ letters acts on $\V$ by permuting the basis, via $w
\cdot \vv_j = \vv_{w(j)}$ extended linearly. This action extends to a
``diagonal'' action on the $r$th tensor power $\V^{\otimes r} = \V
\otimes \cdots \otimes \V$, by
\begin{equation}
w \cdot (\vv_{j_1} \otimes \vv_{j_2} \otimes \cdots \otimes \vv_{j_r})
= \vv_{w(j_1)}\otimes \vv_{w(j_2)} \otimes \cdots \otimes \vv_{w(j_r)}
\end{equation}
for any $j_1, j_2, \dots, j_r \in \nn:= \{1,\dots, n\}$. In other
words, the matrix of the action of $w$, taken with respect to the basis
\begin{equation}
\vv_{i_1} \otimes \vv_{i_2} \otimes \cdots \otimes \vv_{i_r} \quad
(i_1,\dots, i_r \in \nn),
\end{equation}
is the $r$th Kronecker power $P(w)^{\otimes r}$ of the $n \times n$
permutation matrix $P(w):= [\delta_{i,w(j)}]$. Extending the action
linearly to the group algebra $\Bbbk[W_n]$ makes $\V^{\otimes r}$ into
a $\Bbbk[W_n]$-module. Identify $\End_\Bbbk(\V^{\otimes r})$ with the
algebra $\textsf{M}_{n^r}(\Bbbk)$ of $n^r \times n^r$ matrices, via the
basis, and let
\begin{equation}\label{e:Phi}
\Phi: \Bbbk[W_n] \to \End_\Bbbk(\V^{\otimes r}), \quad \textstyle
\sum_{w\in W_n} a_w\, w \mapsto \sum_{w \in W_n} a_w P(w)^{\otimes r}
\end{equation}
be the corresponding matrix representation. The image $\im(\Phi)$
is the $\Bbbk$-linear span of the set $\Gamma$ of $r$th
Kronecker powers of $n \times n$ permutation matrices.  Our main
result is the following.

\begin{thm}\label{t1}
  The set of all $P(w)^{\otimes r}$ such that $w$ is a permutation in
  $W_n$ and the sequence $(w(1),w(2),\dots,w(n))$ has an increasing
  (resp., decreasing) subsequence of length $n-r$ is a $\Bbbk$-basis
  for $\im(\Phi)$.
\end{thm}

The result is of interest only for $r< n-1$, as the subsequence
condition is vacuous for $r \ge n-1$ (in which case $\Phi$ is
faithful). In \cite{BDM:canon}, it is shown that the problem of
expressing an arbitrary element of $\im(\Phi)$ as a linear combination
of the basis in Theorem~\ref{t1} reduces to inverting a
$(0,1)$-unitriangular matrix.

In Section \ref{s1} we assume that $\Bbbk = \R$, the field of real
numbers, and consider nonnegative matrices in $\im(\Phi)$ in the
spirit of \cites{Marcus-Minc, Ryser, Minc:permanents, Minc:non-neg,
  Bapat-Raghavan, BR, Brualdi:06}.  Of particular interest is
Birkhoff's theorem \cite{Birkhoff} (see also von Neumann \cite{vN}),
that the set of $n \times n$ doubly stochastic matrices is the convex
hull of the set of $n \times n$ permutation matrices. We wondered
whether Birkhoff's theorem extends to the set $\Omega$ of $n^r \times
n^r$ doubly stochastic matrices in $\im(\Phi)$; indeed, we conjectured
that it does so extend in an earlier version of this paper. The
conjecture is false, by a recent counterexample (Example \ref{ex1})
due to Roberson and Schmidt.  In Section \ref{s1}, we prove:
\begin{enumerate}\renewcommand{\labelenumi}{(\roman{enumi})}
\item $\Omega$ is convex and the points of $\Gamma$ are extremal
  points of $\Omega$.
\item Birkhoff's theorem extends to $\Omega$ if and only if a theorem
  of K\"{o}nig extends to $\Omega$, and it does so extend if $r \ge
  n-1$.
\end{enumerate}
The interesting question of determining the convex structure of
$\Omega$, and in particular finding its set of vertices, is
highlighted.

The rest of the paper takes place over a general unital commutative
ring $\Bbbk$, unless indicated otherwise.  Section \ref{s2} looks at
$r=1$ in detail; the bases in Theorem \ref{t1} appear to be new even
in that case, and we show that they are indexed by the set of
``consecutive'' cycles in $W_n$. Section \ref{s3} contains the proof
of Theorem~\ref{t1}. Although the proof is straightforward it is heavy
on technical notation; in particular we need to work in the
Iwahori--Hecke algebra of $W_n$ in most of that section. Note that
Schur--Weyl duality is not needed to prove Theorem~\ref{t1}. Theorem
\ref{t2} in Section~\ref{s3} obtains a new ``Kazhdan--Lusztig'' basis
for the annihilator of a certain key permutation module, which may be
of independent interest. Section \ref{s4} explains the connections to
integral Schur--Weyl duality for partition algebras, proved in
\cites{BDM,Donkin}, and shows for instance that $\im(\Phi)
= \End_{\Ptn(r,n)}(\V^{\otimes r})$, the centralizer algebra for the
usual action of the partition algebra on tensor space.  Section
\ref{s5} applies Schur--Weyl duality to describe $\im(\Phi)$ by an
explicit linear system, which relates back to doubly stochastic
matrices in case $\Bbbk=\R$.

\section{Convexity and doubly stochastic matrices}\label{s1}
\noindent
Assume that $\Bbbk=\R$ in this section.  Recall that a
square matrix is \emph{doubly stochastic} if its entries are
nonnegative real numbers and all its rows and columns sum to $1$.

\begin{prop}\label{p1}
  If $\Bbbk=\R$, the set $\Omega$ of doubly stochastic matrices in
  $\im(\Phi)$ is convex, and every $P(w)^{\otimes r}$ is an extremal
  point (vertex) of $\Omega$.
\end{prop}

\begin{proof}
Let $M,M'$ be in $\Omega$, the set of doubly stochastic elements of
$\im(\Phi)$, the $\R$-span $\Gamma = \{P(w)^{\otimes r} \mid w \in
W_n\}$. Then for any $0 \le t \le 1$, $tM + (1-t)M'$ is again doubly
stochastic, and it is also in the $\R$-span of $\Gamma$, so it is in
$\Omega$. Thus $\Omega$ is convex.

Furthermore, if $M = P(w)^{\otimes r}$ is not an extremal point of
$\Omega$, then it must be the midpoint $M = \frac{1}{2}(A+B)$ of the
line segment between two \emph{distinct} points $A,B$ of
$\Omega$. Hence
\[
m_{i_1\cdots i_r,\, j_1\cdots j_r} = \tfrac{1}{2}(a_{i_1\cdots i_r,\,
  j_1\cdots j_r} + b_{i_1\cdots i_r,\, j_1\cdots j_r})
\]
for all $i_1\cdots i_r,\, j_1\cdots j_r$ in $\nn^r$, where each
\[
0 \le a_{i_1\cdots i_r,\, j_1\cdots j_r}, b_{i_1\cdots i_r,\, j_1\cdots
  j_r} \le 1.
\]
If $m_{i_1\cdots i_r,\, j_1\cdots j_r} = 0$ then $a_{i_1\cdots i_r,\,
  j_1\cdots j_r} = b_{i_1\cdots i_r,\, j_1\cdots j_r} = 0$. If
$m_{i_1\cdots i_r,\, j_1\cdots j_r} = 1$ then $a_{i_1\cdots i_r,\,
  j_1\cdots j_r} = b_{i_1\cdots i_r,\, j_1\cdots j_r} = 1$. Since
$0,1$ are the only possible values of the entries of $M$, we see that
$A=B$, which is the desired contradiction.
\end{proof}

In case $r=1$, Birkhoff's theorem \cite{Birkhoff} characterizes the
set of $n \times n$ doubly stochastic matrices as the convex hull of
the set of $n \times n$ permutation matrices. In light of Proposition
\ref{p1}, it is natural to ask the following question.

\begin{quest}\label{q1}
  Does Birkhoff's result extend to the set $\Omega$, for all $n$, $r$?
  In other words, is $\Omega$ the convex hull of $\Gamma$, for all
  $n,r$?
\end{quest}

Proposition \ref{p3} shows that the answer is yes for all $r$
sufficiently large, but it is no in general, as we will see. Recall
that Birkhoff's theorem is implied by a theorem of K\"{o}nig
\cite{Koenig}, which states that any $n \times n$ doubly stochastic
matrix $M=[m_{ij}]$ has a positive diagonal, where a \emph{diagonal}
is defined to be $\{m_{w(j),j}\}_{j\in [n]}$, for some $w \in W_n$
(the entries corresponding to the nonzero entries in
$P(w)$). Equivalently, K\"{o}nig's result is that the permanent of any
doubly stochastic matrix is positive; recall that the permanent is the
sum of all diagonal products.

\begin{prop}\label{p2}
  Let $\Bbbk=\R$. For any $n,r$ the following are equivalent:
  \begin{enumerate}
  \item The set $\Omega$ of doubly stochastic matrices in $\im(\Phi)$
    is equal to the convex hull of $\Gamma = \{ P(w)^{\otimes r} \mid w
    \in W_n \}$.

  \item Every $M = [m_{i_1\cdots i_r,\, j_1\cdots j_r}]$ in $\Omega$
    has a positive ``Kronecker power'' diagonal; that is, a diagonal,
    with all entries positive, of the form
    \[
    \{ m_{w(j_1)w(j_2) \cdots w(j_n),\, j_1j_2 \cdots j_r} \mid j_1, j_2,
    \dots, j_r = 1, \dots, n \}
    \]
    corresponding to the nonzero entries in $P(w)^{\otimes r}$, for
    some $w \in W_n$.
  \end{enumerate}
\end{prop}

\begin{proof}
This is an extension to higher Kronecker powers of standard arguments,
e.g., \cite{Marcus-Minc}*{II.1.7}, \cite{Bapat-Raghavan}*{Thm.~2.1.4},
or \cite{BR}*{Thm.~1.2.1}.

We first show that (b) implies (a). Assume that (b) holds. Any convex
linear combination of the $P(w)^{\otimes r}$ clearly belongs to
$\Omega$, so we only need to show the reverse inclusion. Let $M$ be in
$\Omega$. By (b), there is a positive Kronecker power diagonal in $M$,
indexed by some $w \in W_n$. Let $c_w$ be the minimum entry in that
diagonal.  If $c_w=1$ then $M=P(w)^{\otimes r}$ and we are
done. Otherwise $0 < c_w < 1$ and the matrix $M' := \frac{1}{1-c_w}(M
- c_w P(w)^{\otimes r})$ is again in $\Omega$. Note that $M'$ has at
least one more zero entry than $M$, and $M = c_w P(w)^{\otimes r} +
(1-c_w) M'$. We then repeat the argument with $M'$ in place of
$M$. The process must terminate in finitely many steps, as the number
of nonzero entries in the sequence of matrices forms a strictly
decreasing sequence. Upon termination, we have found real scalars $c_w
\ge 0$ such that
\[
M = \textstyle \sum_{w \in W_n} c_w P(w)^{\otimes r}
\quad\text{and}\quad \sum_{w \in W_n} c_w = 1
\]
which is a convex linear combination, thus proving the reverse
inclusion and the desired equality.

Conversely, the fact that (a) implies (b) is immediate, as the diagonal
corresponding to any nonzero summand in a convex linear combination of
the $P(w)^{\otimes r}$ must be positive.
\end{proof}

\begin{prop}\label{p3}
  Assume that $\Bbbk=\R$. If $r \ge n-1$ then $\Omega$ is equal to the
  convex hull of the set $\Gamma = \{P(w)^{\otimes r} \mid w \in
  W_n\}$. In other words, the analogue of Birkhoff's theorem holds.
\end{prop}

\begin{proof}
By \cite{BDM:kernel}*{Cor.~4.13}, the representation $\Phi$ is
injective for any $r \ge n-1$, hence induces an isomorphism $\R[W_n]
\cong \im(\Phi)$. Thus there is always a unique solution to the
equation
\[
\textstyle \Phi\big(\sum_{w \in W_n} c_w w\big) = \sum_{w \in W_n} c_w
P(w)^{\otimes r} = M
\]
for any given $M$ in $\im(\Phi)$. We have
\[
m_{i_1\cdots i_r,\, j_1\cdots j_r} = \textstyle \sum_{w\in W_n:
  w(j_\alpha)=i_\alpha,\, \forall \alpha} c_w.
\]
If $r \ge n$, only one $P(w)^{\otimes r}$ can contribute to any
$m_{w(j_1)\cdots w(j_r), j_1\cdots j_r}$, where there are exactly $n$
distinct values (the maximum possible) in $\{j_1, \dots, j_r\}$. So we
must take $c_w = m_{w(j_1)\cdots w(j_r), j_1\cdots j_r}$ equal to that
entry of $M$, for each $w$. Thus, if $M$ happens to be doubly
stochastic, then each $c_w \ge 0$. At least one of the $c_w$ is
positive, and the corresponding diagonal is a positive Kronecker power
diagonal in $M$, so by Proposition \ref{p2}, Birkhoff's theorem holds
in this case.

If $r = n-1$ then the same reasoning applies to any $m_{w(j_1)\cdots
  w(j_{n-1}), j_1\cdots j_{n-1}}$, where the values in $\{j_1, \dots,
j_{n-1}\}$ are all distinct (we can take $j_k = k$ here, for
instance). The point is that any permutation of $n$ objects is
determined by its values on $n-1$ of them. So the rest of the argument
goes through as in the preceding paragraph.
\end{proof}

However, in general the answer to Question \ref{q1} is no, as shown in
the following simple counterexample, based on Roberson and Schmidt
\cite{RS}*{Sect.~3}. To set the stage, we observe that $P(w)^{\otimes
  2}$ is the block matrix $[\delta_{i,w(j)} P(w)]$. In other words, it
has a copy of $P(w)$ in each block corresponding to a $1$-entry of
$P(w)$, and all other blocks are zero.  For instance, if $t$ is the
transposition that interchanges $(1,2)$ in $W_4$ then
\[
P(t)^{\otimes 2} =
\begin{bmatrix}
  0 & P(t) & 0 & 0\\
  P(t) & 0 & 0 & 0\\
  0 & 0 & P(t) & 0\\
  0 & 0 & 0 & P(t)
\end{bmatrix}
\]
as a block matrix. Now we are ready for the promised example.

\begin{example}[Roberson--Schmidt]\label{ex1}
Assume that $(n,r) = (4,2)$, and let $M = \sum_{w \in W_4} c_w
P(w)^{\otimes 2}$, where $c_w = 1/5$ for each transposition $w \in
W_4$, $c_1 = -1/5$ is the coefficient of the identity matrix, and $c_w
= 0$ for all other $w \in W_4$. Then $M$ has the block form shown in
Figure~\ref{fig1},
\begin{figure}[ht]
  \caption{Counterexample for $(n,r)=(4,2)$}
\small
\[ 
\left[\begin{array}{rrrr|rrrr|rrrr|rrrr}\label{fig1}
.4 & 0 & 0 & 0 & 0 & .2 & 0 & 0 & 0 & 0 & .2 & 0 & 0 & 0 & 0 & .2 \\
0 & 0 & .2 & .2 & .2 & 0 & 0 & 0 & 0 & .2 & 0 & 0 & 0 & .2 & 0 & 0 \\
0 & .2 & 0 & .2 & 0 & 0 & .2 & 0 & .2 & 0 & 0 & 0 & 0 & 0 & .2 & 0 \\
0 & .2 & .2 & 0 & 0 & 0 & 0 & .2 & 0 & 0 & 0 & .2 & .2 & 0 & 0 & 0 \\
\hline
 0 & .2 & 0 & 0 & 0 & 0 & .2 & .2 & .2 & 0 & 0 & 0 & .2 & 0 & 0 & 0 \\
.2 & 0 & 0 & 0 & 0 & .4 & 0 & 0 & 0 & 0 & .2 & 0 & 0 & 0 & 0 & .2 \\
0 & 0 & .2 & 0 & .2 & 0 & 0 & .2 & 0 & .2 & 0 & 0 & 0 & 0 & .2 & 0 \\
0 & 0 & 0 & .2 & .2 & 0 & .2 & 0 & 0 & 0 & 0 & .2 & 0 & .2 & 0 & 0 \\
\hline
 0 & 0 & .2 & 0 & .2 & 0 & 0 & 0 & 0 & .2 & 0 & .2 & .2 & 0 & 0 & 0 \\
0 & .2 & 0 & 0 & 0 & 0 & .2 & 0 & .2 & 0 & 0 & .2 & 0 & .2 & 0 & 0 \\
.2 & 0 & 0 & 0 & 0 & .2 & 0 & 0 & 0 & 0 & .4 & 0 & 0 & 0 & 0 & .2 \\
0 & 0 & 0 & .2 & 0 & 0 & 0 & .2 & .2 & .2 & 0 & 0 & 0 & 0 & .2 & 0 \\
\hline
 0 & 0 & 0 & .2 & .2 & 0 & 0 & 0 & .2 & 0 & 0 & 0 & 0 & .2 & .2 & 0 \\
0 & .2 & 0 & 0 & 0 & 0 & 0 & .2 & 0 & .2 & 0 & 0 & .2 & 0 & .2 & 0 \\
0 & 0 & .2 & 0 & 0 & 0 & .2 & 0 & 0 & 0 & 0 & .2 & .2 & .2 & 0 & 0 \\
.2 & 0 & 0 & 0 & 0 & .2 & 0 & 0 & 0 & 0 & .2 & 0 & 0 & 0 & 0 & .4
\end{array}\right]
\]\normalsize
\end{figure}
with its rows and columns indexed by $[4] \times [4]$ ordered
lexicographically. This matrix is doubly stochastic.  We claim that it
contains no positive Kronecker power diagonal, and thus by Proposition
\ref{p2} does not lie in the convex hull of $\Gamma$. Notice that
\[
M_{i,j} = .2 P(t_{i,j}) \qquad\text{for all } i \ne j
\]
where $t_{i,j}$ is the transposition interchanging $i$, $j$. Each of
these blocks has a unique positive diagonal. By the observation
preceding this example, if $M$ had a positive Kronecker power
diagonal, it would be of the form $P(t)^{\otimes 2}$, for some
transposition $t$. But none of the blocks $M_{i,i}$ on the main block
diagonal contains a positive diagonal corresponding to any
transposition, so the claim is established.
\end{example}

We note that Proposition \ref{p1} implies that the analogue of
Birkhoff's theorem holds for a given $r$ if and only if the
$P(w)^{\otimes r}$, $w \in W_n$, are the \emph{only} vertices
of the convex region $\Omega$. The existence of Example \ref{ex1}
suggests the following interesting problem.

\begin{quest}
  Determine the vertex set of the convex polytope $\Omega$.
\end{quest}

By Proposition \ref{p1}, all the points in $\Gamma$ are vertices of
$\Omega$, but Example \ref{ex1} shows that there can be others. In
fact, for $(n,r)=(4,2)$, calculations by Roberson and Schmidt reveal
that $\Omega$ has 162 vertices, far more than the 24 in $\Gamma$.

\section{Interpretation of the main result in case $r=1$}\label{s2}
\noindent
We work over $\Bbbk$ (an arbitrary unital commutative ring) from now
on, unless explicitly stated otherwise. Theorem \ref{t1} gives two
bases of $\im(\Phi)$ which appear to be new, even for $r=1$. We
wish to explore this case in detail, as it turns out that the set of
permutations in question has interesting structure.

We need to understand the set of $w \in W_n$ having an increasing
subsequence of length $n-1$. It is easy to list all such $w$ by a
combinatorial process of filling in $n$ slots. We will use the
shorthand notation $w_1w_2 \cdots w_n$ for the sequence $(w(1), w(2),
\dots, w(n))$ for $w \in W_n$.  To construct a permutation $w_1w_2
\cdots w_n$ on the list, that is, one having an increasing subsequence
of length $n-1$, pick a number $k \in \nn$ and a slot $j \in \nn$, and
place $k$ in the $j$th slot. The remaining elements, i.e., those in
$\nn \setminus \{k\}$, are placed in the remaining slots in
increasing order. As there are $n$ choices for the number and $n$
choices for its slot, there are $n^2$ items in the list.

\begin{example}\label{ex2}
  If $n=4$, carrying out the above procedure yields the following grid
  of sequences, in the shorthand notation:
\[
\begin{array}{llll}
  \ul{1}234 & 2\ul{1}34 & 23\ul{1}4 & 234\ul{1} \\
  \ul{2}134 & 1\ul{2}34 & 13\ul{2}4 & 134\ul{2} \\
  \ul{3}124 & 1\ul{3}24 & 12\ul{3}4 & 124\ul{3} \\
  \ul{4}123 & 1\ul{4}23 & 12\ul{4}3 & 123\ul{4}
\end{array}
\]
in which we have underlined the number placed in the chosen slot.
\end{example}

Notice that the identity permutation appears $n$ times on the main
diagonal, and the $n-1$ elements on the superdiagonal are the same as
the corresponding ones on the subdiagonal. So our list overcounts by
$2(n-1)$ items. Omitting the duplicates, we obtain a list of $n^2 - 2n
+ 2$ permutations, which is the (well known) dimension of $\im(\Phi)$
in the $r=1$ case.

The structure of this set of permutations is revealed by writing the
permutations not as sequences, but instead as products of disjoint
cycles.

\begin{example}\label{ex3}
The corresponding elements in Example \ref{ex2} written in the
\emph{cycle notation} (e.g., $(4,3,2)$ means $4 \mapsto 3 \mapsto 2
\mapsto 4$) are:
\[
\begin{array}{cccc}
  (1) & (1,2) & (1,2,3) & (1,2,3,4) \\
  (2,1) & (1) & (2,3) & (2,3,4) \\
  (3,2,1) & (3,2) & (1) & (3,4) \\
  (4,3,2,1) & (4,3,2) & (4,3) & (1)
\end{array}
\]
where we write $(1)$ for the identity permutation.
\end{example}

Observe that every element consists of a single cycle of consecutive
numbers, and all such cycles appear. Elements along the diagonals have
the same cycle length, the cycle length increasing by one each step as
the diagonal distance increases away from the main diagonal. Cycles
which are in symmetric positions about the diagonal are mutual
inverses. Finally, we observe that the picture is compatible with
restriction, because we obtain the grid for $n-1$ by deleting the last
row and column of the grid for $n$.

Let $w_0$ be the longest element (with respect to the usual Coxeter
length function) in $W_n$. Note that $w_0$, as a sequence, is the reverse of
the identity sequence (it swaps $(1,n)$, $(2,n-1)$, etc).

We define a \emph{consecutive} $k$-cycle to be either a $k$-cycle
which maps each integer in the interval $[i,i+k-1]$ to its successor
modulo $k$, or its inverse. Notice that all the cycles in Example
\ref{ex3} are consecutive.

\begin{prop}\label{p4}
  The set of $w \in W_n$ having an increasing subsequence of length
  $n-1$ is the same as the set $\CC(n)$ of all consecutive cycles in
  $W_n$, and thus $\{P(w) \mid w \in \CC(n)\}$, $\{P(ww_0) \mid w \in
  \CC(n)\}$ are the bases in Theorem \ref{t1} in case $r=1$.
\end{prop}

\begin{proof}
The first claim is proved by induction on $n$. Assuming the desired
equality in the statement has been established for $n-1$, one easily
checks that the additional $2n-3$ consecutive cycles which move $n$
coincide with the non-diagonal elements in the last row and column of
the grid, which shows that the equality holds when $n-1$ is replaced
by $n$.

Hence $\{P(w) \mid w \in \CC(n)\}$ is the first basis in
Theorem~\ref{t1}.  The second basis is obtained by reversing the order
of each sequence in the grid, which obviously interchanges increasing
and decreasing subsequences. Let $\rev{w}$ denote the reverse of $w
\in W_n$.  Then $\rev{w}= ww_0$.  It follows that $\{P(ww_0) \mid w
\in \CC(n)\}$ is the second basis in Theorem~\ref{t1}.
\end{proof}

The two bases in Proposition \ref{p4} appear to be new. Compare
e.g.\ with the bases in
\cites{Farahat-Mirsky,Farahat,Gibson,Johnsen,Lai,Brualdi-Cao}.  The
first basis in Proposition \ref{p4} has something of the same flavor
as that in \cite{Farahat-Mirsky}*{Cor.~1}, which is also indexed by a
(different) set of cycles and is also compatible with restriction.

\section{Proof of Theorem \ref{t1}}\label{s3}\noindent
The main task of the proof is to rewrite the basis of $\ker(\Phi)$
given in \cite{BDM:kernel}*{Thm.~7.4} (which is written in terms of
certain Murphy basis elements at $v=1$) in terms of the
Kazhdan--Lusztig basis.  Our main technical tool is the paper of Geck
\cite{Geck}, which works out the relation between the two approaches
in the context of the Iwahori--Hecke algebra $\Hec$ associated to the
symmetric group. One may wish to compare our proof with the proof of
\cite{RSS}*{Thm.~1}, which is also based on Geck's paper, although
both the results and proofs are different. 

Let $S$ be the set of adjacent transpositions in $W_n$ and write $W =
W_n$ in this section. Let $l:W \to \Z_{\ge 0}$ be the usual %Coxeter
length function with respect to $S$. The Iwahori--Hecke algebra $\Hec
= \Hec(W,S)$ is the $\LP$-algebra ($v$ an indeterminate) with basis
$\{T_w \mid w \in W\}$ (where $T_1 = 1$) and with multiplication given
by
\[
T_s T_w =
\begin{cases}
  T_{sw} & \text{ if } l(sw) = l(w)+1,\\
  T_{sw}+(v-v^{-1}) T_w & \text{ if } l(sw) = l(w)-1
\end{cases}
\]
for all $w \in W$, $s \in S$.

\begin{rmk}\label{r1}
We follow the notational conventions of \cites{Lu,Geck} here; in
particular the generators $T_s$ satisfy the ``balanced'' quadratic
eigenvalue relation $(T_s+v^{-1})(T_s-v) = 0$. To get back to the
setup in the older articles \cites{Murphy:92,Murphy:95,KL} one needs
to set $q = v^2$ and replace $T_s$ by $v T_s$ (which defines an
isomorphism between the two versions).
\end{rmk}

There is a unique ring involution $\LP \to \LP$, written $a \mapsto
\ov{a}$, such that $\ov{v}=v^{-1}$. This extends to a ring involution
$\jmath: \Hec \to \Hec$ such that
\begin{equation}\label{e:jmap}
\textstyle \jmath\left(\sum_{w \in W} a_w T_w\right) = \sum_{w\in W}
(-1)^{l(w)} \ov{a}_w T_w
\end{equation}
for any $a_w \in \LP$. There is also a unique $\LP$-algebra
automorphism
\begin{equation}
\dagger: \Hec \to \Hec \text{ such that } T_s \mapsto T_s^\dagger =
-T_s^{-1} \quad (s \in S).
\end{equation}
We have $T_w^\dagger = (-1)^{l(w)} T_{w^{-1}}^{-1}$, for any $w \in
W$.  The maps $\jmath, \dagger$ commute. Define a map
$\ov{\phantom{a}}: \Hec \to \Hec$, $h \mapsto \ov{h}$, where $\ov{h} =
\jmath(h^\dagger) = \jmath(h)^\dagger$. The map $\ov{\phantom{a}}$ is a
ring involution of $\Hec$ such that
\begin{equation}
\textstyle \ov{\sum_{w \in W} a_wT_w} = \sum_{w \in W} \ov{a}_w
T_{w^{-1}}^{-1} \quad (a_w \in \LP).
\end{equation}
By \cite{Lu}*{Thm.~5.2}, for any $w \in W$, there exist unique $C_w$,
$C'_w$ in $\Hec$ such that
\begin{equation}
\begin{aligned}
  \ov{C_w} &= C_w \quad\text{and}\quad C_w \equiv T_w \mod{\Hec_{>0}}\\
  \ov{C'_w} &= C'_w \quad\text{and}\quad C'_w \equiv T_w \mod{\Hec_{<0}}
\end{aligned}
\end{equation}
where
\[
  \Hec_{>0} := \textstyle \sum_{w \in W} v\Z[v] T_w, \qquad
  \Hec_{<0} := \textstyle \sum_{w \in W} v^{-1}\Z[v^{-1}]T_w .
\]
Then $\{C_w \mid w \in W\}$, $\{C'_w \mid w \in W\}$ are both bases of
$\Hec$. These are the ``Kazhdan--Lusztig bases'' first introduced in
\cite{KL}. It was proved in \cite{KL}*{Thm.~1.1} that
\begin{equation}\label{eq:KL}
C_w = T_w + \sum_{y<w} (-1)^{l(w)+l(y)} \ov{p_{y,w}} \, T_y, \qquad
C'_w = T_w + \sum_{y<w} p_{y,w} \, T_y
\end{equation}
for any $w \in W$, where both sums are over the set of $y \in W$ such
that $y<w$ in the Bruhat--Chevalley order on $W=W_n$ and $p_{y,w}$ is
in $v^{-1}\Z[v^{-1}]$. It follows that $C_w = (-1)^{l(w)}
\jmath(C'_w)$.

Now we recall Murphy's bases.  As usual, we write $\lambda \vdash n$
to indicate that $\lambda = (\lambda_1, \lambda_2, \dots, \lambda_k)$
is a partition of $n$ (meaning that $\lambda \in \Z^k$, $\lambda_1 \ge
\lambda_2 \ge \cdots \ge \lambda_k >0$, and $\sum \lambda_i = n$).  If
$\lambda \vdash n$, set
\[
\Tab(\lambda) = \{ \text{standard Young tableaux of shape } \lambda \}
\]
where as usual the numbers in a standard tableau are strictly
increasing along the rows and down the columns. (See \cite{Fulton} for
details.) In \cite{Murphy:92} (see also \cite{Murphy:95}) Murphy
introduces two bases
\[
\{x_{\fs\ft} \mid \fs,\ft \in \Tab(\lambda), \lambda \vdash n\},\quad
\{y_{\fs\ft} \mid \fs,\ft \in \Tab(\lambda), \lambda \vdash n\}
\]
of $\Hec$, indexed by pairs of standard Young tableaux of the same
shape. For $\lambda \vdash n$ and $\fs,\ft \in \Tab(\lambda)$,
\begin{equation}
\begin{aligned}
  x_\lambda &:= \textstyle\sum_{w \in W_\lambda} v^{l(w)} T_w
  \quad\text{and}\quad x_{\fs\ft} := T_{d(\fs)} x_\lambda
  T_{d(\ft)^{-1}} \\ y_\lambda &:= \textstyle\sum_{w \in W_\lambda}
  (-v)^{-l(w)} T_w \quad\text{and}\quad y_{\fs\ft} := T_{d(\fs)}
  y_\lambda T_{d(\ft)^{-1}}
\end{aligned}
\end{equation}
where $W_\lambda$ is the usual Young subgroup associated to $\lambda$
and $d(\ft) := y$, for a tableau $\ft$, if $y \in W$ is the unique
element of $W$ such that $y \ft = \ft^\lambda$. Here $\ft^\lambda$ is
the tableau in which the the numbers $1, \dots, n$ appear in their
natural order, written as in order across rows from the top row to the
bottom one. Notice that $W_\lambda$ is the row-stabilizer of
$\ft^\lambda$.

\begin{rmk}\label{r2}
The notation here differs slightly from Murphy's in
\cites{Murphy:92,Murphy:95}. Because of renormalization (see Remark
\ref{r1}) what he writes as $T_w$ corresponds to $v^{l(w)}T_w$ in our
notation. Also, the order of the products defining $x_{\fs\ft}$,
$y_{\fs\ft}$ is reversed here, as we deal with left modules while he
works with right ones. Our conventions are chosen to agree with those
in Geck's paper \cite{Geck}.
\end{rmk}

Recall \cites{Stanley,Fulton} or \cite{Knuth}*{\S5.1.4} that the
Robinson--Schensted--Knuth (RSK for short) correspondence gives a
bijection
\[
\bigsqcup_{\lambda \vdash n} \big(\Tab(\lambda) \times
\Tab(\lambda)\big) \to W
\]
mapping pairs of standard tableaux of the same shape to permutations.
Write $\pi_\lambda(\fs,\ft)$ for the image of a pair $(\fs,\ft)$ of
standard tableaux of shape $\lambda$. Given $w \in W$, the pair
$(\fs,\ft)$ such that $\pi_\lambda(\fs,\ft) = w$ is explicitly
constructed by the insertion algorithm \cites{Fulton,Stanley,Knuth}:
$\fs$ is obtained by inserting the numbers in the sequence $(w(1),
w(2), \dots, w(n))$ into an initially empty tableau, and $\ft$ records
the order in which the positions of $\fs$ were filled.

For any $\fs,\ft \in \Tab(\lambda)$, any $\lambda \vdash n$, the ring
involution $\jmath: \Hec \to \Hec$ defined in \eqref{e:jmap} satisfies
\begin{equation}
  \jmath(x_\lambda)=y_\lambda \text{ and thus }
  \jmath(x_{\fs\ft})= \pm y_{\fs\ft} .
\end{equation}
Hence by \cite{Geck}*{Cor.~4.3}, it follows that Geck's element
$\tilde{y}_{\fs\ft}:= T_{d(\fs)} C_{w_\lambda} T_{d(\ft)^{-1}}$
satisfies the identity
\begin{equation}\label{e10}
  \tilde{y}_{\fs\ft} = \pm v^{l(w_\lambda)} y_{\fs\ft}.
\end{equation}
The element $w_\lambda = \pi_{\lambda'}(\ft^{\lambda'},
\ft^{\lambda'})$ here is the longest word in $W_\lambda$, where
\[
\text{$\lambda'$ denotes the transpose partition of $\lambda$.}
\]
By \cite{Geck}*{Cor.~5.6}, the two-sided Kazhdan--Lusztig cell
$\cell(\lambda)$ indexed by any $\lambda \vdash n$ is given by
\begin{equation}
  \cell(\lambda) = \{\pi_{\lambda'}(\fs,\ft) \mid \fs,\ft \in
  \Tab(\lambda') \}.
\end{equation}
For any $w \in W$, Geck writes $\lambda_w = \lambda \iff w \in
\cell(\lambda)$. We prefer instead to label cells by their
\emph{RSK-shape}, written $\RSK(w)$, which we define to be the common
shape of the associated pair of tableaux in the
RSK-correspondence. Thus
\begin{equation}
  \lambda_w = \lambda \iff \RSK(w) = \lambda'; \quad\text{i.e., }
  \RSK(w) = \lambda'_w .
\end{equation}
Recall the \emph{dominance order} $\dom$ on partitions, a partial
order, defined by
\[
  \lambda = (\lambda_1, \dots, \lambda_m) \dom \mu = (\mu_1, \dots,
  \mu_l)
\]
if $\sum_{i\le k} \lambda_i \ge \sum_{i\le k} \mu_i$ for all
$k$. Write $\lambda \lessdom \mu$ if and only if $\mu \dom \lambda$,
and $\lambda \slessdom \mu$ if and only if $\lambda \lessdom \mu$ but
$\lambda\ne \mu$, etc. Recall that transposition reverses the
dominance order: $\lambda \dom \mu \iff \lambda' \lessdom \mu'$.
Reformulating the statement of \cite{Geck}*{Cor.~4.11} in light of
\cite{Geck}*{Cor.~5.11} in these terms gives the following.

\begin{prop}[Geck]\label{p:Geck}
  Let $\lambda \vdash n$. For any $\fs,\ft$ in $\Tab(\lambda)$, there
  exists a unique element $w \in \cell(\lambda)$ of RSK-shape
  $\lambda'$, such that
  \[
  \tilde{y}_{\fs\ft} = C_w \quad + \sum_{x :\, \RSK(x)=\lambda'} a_xC_x
  \quad + \sum_{x :\, \RSK(x)\slessdom \lambda'} b_xC_x
  \]
  where $a_x \in v\Z[v]$, $b_x \in \LP$ for all $x$.
\end{prop}

Recall what it means to ``specialize $v \mapsto \xi$ in $\Bbbk$''. If
$\xi \in \Bbbk$ is invertible, we regard $\Bbbk$ as a $\LP$-algebra by
means of the (unique) ring homomorphism $\LP \to \Bbbk$ sending
$v^{\pm 1} \mapsto \xi^{\pm 1}$. Let $\Hec_\Bbbk := \Bbbk
\otimes_{\LP} \Hec$ be the $\Bbbk$-algebra obtained by extending
scalars via this morphism. By abuse of notation, we identify symbols
such as $T_w$, $C_w$, $y_{\fs\ft}$, etc with their respective images
$1\otimes T_w$, $1\otimes C_w$, $1\otimes y_{\fs\ft}$, etc in
$\Hec_\Bbbk$.  As in Dipper and James \cite{DJ}, the left ideal
$M^\lambda := \Hec x_\lambda$ is a ``permutation module'' indexed by
$\lambda \vdash n$. If we specialize $v \mapsto 1$ in $\Bbbk$ then
$\Hec_\Bbbk \cong \Bbbk[W_n]$ and $M^\lambda$ is isomorphic to the
usual permutation module for $\Bbbk[W_n]$.

At this point, there are two possibilities for how to proceed,
depending on whether we prefer to specialize now or later. Rather than
favor one over the other, we discuss both. 

\begin{thm}\label{t2}
Suppose that $r<n-1$. Let $\alpha(n,r) := (n-r,1^r)$ be the partition
$(n-r,1,\dots, 1)$ with $1$ repeated $r$ times.
\begin{enumerate}
\item Under specialization $v \mapsto 1$ in $\Bbbk$, the set $\{ C_x
  \mid \RSK(x) \notdom \alpha(n,r) \}$ is a $\Bbbk$-basis of the
  annihilator of the $\Bbbk[W_n]$-action on $\V^{\otimes r}$.

\item Over $\LP$, the set $\{ C_x \mid \RSK(x) \notdom \alpha(n,r) \}$
  is a $\Z[v,v^{-1}]$-basis of the annihilator of the $\Hec$-action
  on $M^{\alpha(n,r)}$.
\end{enumerate}
\end{thm}

\begin{proof}
(a)
Let $\Delta := \{ a \in \Bbbk[W_n] \mid a \text{ annihilates } \V^{\otimes
  r}\}$. By \cite{BDM:kernel}*{Thm.~7.4}, the set
\[
\{ y_{\fs\ft} \mid \fs,\ft \in \Tab(\lambda),\ \lambda \notlessdom
\alpha(n,r)' \} = \{ y_{\fs\ft} \mid \fs,\ft \in
\Tab(\lambda),\ \lambda' \notdom \alpha(n,r) \}
\]
is a $\Bbbk$-basis of $\Delta$. Let
\[
a  \quad = \sum_{\fs,\ft \in \Tab(\lambda):\, \lambda' \notdom
\alpha(n,r)} a_{\fs\ft} y_{\fs\ft} \quad (a_{\fs\ft} \in \Bbbk)
\]
be an arbitrary element of $\Delta$.  As we are working at $v=1$, we
have $y_{\fs\ft} = \pm\tilde{y}_{\fs\ft}$ by equation \eqref{e10}. By
Proposition \ref{p:Geck}, each $y_{\fs\ft}$ appearing on the right
hand side of the above equality belongs to the $\Bbbk$-span of $\{C_x
\mid \RSK(x) \lessdom \lambda'\}$, for some $\lambda' \notdom
\alpha(n,r)$.  But
\[
\lambda' \notdom \alpha(n,r) \text{ and } \RSK(x) \lessdom \lambda'
\implies \RSK(x) \notdom \alpha(n,r),
\]
so $a$ is in the span of $\{ C_x \mid \RSK(x) \notdom \alpha(n,r) \}$.
So that set spans $\Delta$. Linear independence is clear, so it is a
basis. This proves (a).

(b) Now let $\Delta:= \{ a \in \Hec \mid a \text{ annihilates }
M^{\alpha(n,r)} \}$.  By \cite{Donkin}*{\S8}, which extends the main
result of \cite{BDM:kernel} to $\Hec$, essentially the same set
\[
\{ y_{\fs\ft} \mid \fs,\ft \in \Tab(\lambda),\ \lambda \notlessdom
\alpha(n,r)' \} = \{ y_{\fs\ft} \mid \fs,\ft \in
\Tab(\lambda),\ \lambda' \notdom \alpha(n,r) \}
\]
is a $\LP$-basis of $\Delta$. The rest of the argument is almost exactly
the same as for part (a), except that coefficients are in $\LP$. The
power of $v$ in equation \eqref{e10} causes no trouble, as $v$ is
invertible.
\end{proof}

\begin{rmk} \label{r3}
(i) Parts (a), (b) of Theorem \ref{t2} are connected by
  \cite{BDM:kernel}*{Thm.~7.4(c)}, which says that when $v \mapsto 1$
  in $\Bbbk$, the annihilators of $\V^{\otimes r}$ and
  $M^{\alpha(n,r)}$ coincide. It follows that (b) implies (a) in
  Theorem~\ref{t2}.  On the other hand, we proved (a) directly without
  assuming (b), based on the main result of \cite{BDM:kernel}, and (a)
  is really all we need.  (ii) There is no Hopf algebra structure on
  $\Hec$ properly deforming that of $\Z[W_n]$, so there is no
  interesting ``$q$-analogue'' of the diagonal action of $\Bbbk[W_n]$
  on $\V^{\otimes r}$; thus it doesn't make sense to ask for the
  annihilator of $\V^{\otimes r}$ in the context of part (b).
\end{rmk}

The following Lemma will be applied to deduce the Corollary to
Theorem~\ref{t2} that follows, which in turn is used in proving
Theorem~\ref{t1}.

\begin{lem*}
  For any subset $U$ of $W_n$, the image of $\{T_w \mid w \in W_n
  \setminus U \}$ is a $\LP$-basis of $\Hec/\Delta$, where $\Delta$ is
  the submodule spanned by $\{C_x \mid x \in U\}$. The corresponding
  statement holds over $\Bbbk$ upon specialization $v \mapsto 1$.
\end{lem*}

\begin{proof}
This is essentially the same idea as M\"{o}bius inversion over the
poset $W_n$ under the Bruhat--Chevalley order, using the unitriangular
relation \eqref{eq:KL} between the bases $\{ C_x \mid x \in W_n\}$,
$\{ T_x \mid x \in W_n\}$. By inverting the unitriangular matrix
giving the basis transition in \eqref{eq:KL}, we see that
\[
T_w = C_w + \sum_{y<w} b_{y,w} C_y \quad (b_{y,w} \in \LP).
\]
This implies a similar relation holds in the image $\Hec/\Delta$, that is,
\[
\tilde{T}_w = \tilde{C}_w + \sum_{y<w,\, y \notin U} b_{y,w} \tilde{C}_y
\quad (b_{y,w} \in \LP)
\]
where we set $\tilde{T}_w := T_w + \Delta$, $\tilde{C}_w := C_w +
\Delta$ in the quotient. Clearly, the set $\{ \tilde{C}_x \mid x
\notin U\}$ is a basis of $\Hec/\Delta$. Inverting again, we see that
\[
\tilde{C}_w = \tilde{T}_w + \sum_{y<w,\, y \notin U} d_{y,w} \tilde{T}_y
\quad (d_{y,w} \in \LP).
\]
Thus, $\Hec/\Delta$ is spanned by $\{\tilde{T}_w \mid w \in W_n
\setminus U \}$. We leave the proof of linear independence of that set
to the reader.
\end{proof}

\begin{cor*}
The image of $\{ T_x \in \Hec \mid \RSK(x) \dom \alpha(n,r) \}$ is a
$\LP$-basis of the quotient $\Hec/\Delta$, where $\Delta$ is the
annihilator of $M^{\alpha(n,r)}$. Under specialization $v \mapsto 1$
in $\Bbbk$, the image of $\{ x \in W_n \mid \RSK(x) \dom \alpha(n,r)
\}$ is a $\Bbbk$-basis of $\Bbbk[W_n]/\Delta$, where $\Delta =
\ker(\Phi)$ is the annihilator of $\V^{\otimes r}$.
\end{cor*}

\begin{proof}
The first statement follows from part (b) of Theorem \ref{t2}, by
taking $U$ in the Lemma to be the set of $x$ in $W_n$ such that
$\RSK(x) \notdom \alpha(n,r)$. The second statement follows from the
first, as $T_x$ becomes $x$ upon specialization $v \mapsto
1$. Alternatively, it follows from part (a) of Theorem \ref{t2}, by
making the same choice for $U$ and specializing in the Lemma.
\end{proof}

We can now give the proof of the main result.

\begin{proof}[Proof of Theorem \ref{t1}]
By the Corollary, $\{P(w)^{\otimes r} \mid w \in W_n, \RSK(w) \dom
\alpha(n,r)\}$ is a $\Bbbk$-basis of $\im(\Phi) \cong
\Bbbk[W_n]/\ker(\Phi)$. By \cite{BDM:kernel}*{Lem.~6.2},
\[
\lambda \dom \alpha(n,r) \iff \lambda_1 \ge n-r
\]
and by Schensted's theorem \cite{Schensted} (see also \cite{Stanley})
we have
\[
\RSK(w) = \lambda \implies \IS(w) = \lambda_1 
\]
where we define $\IS(w)$ to be the length of the longest increasing
subsequence of $(w(1), w(2), \dots, w(n))$.
Putting these facts together shows that
\[
\lambda = \RSK(x) \dom \alpha(n,r) \iff \IS(\lambda) \ge n-r
\]
which gives the first basis in Theorem \ref{t1}.

The existence of the second basis in Theorem \ref{t1} follows from the
first, using another observation of Schensted, that
$\DS(w)=\IS(\rev{w})$ for any $w$ in $W_n$, where $\DS(w)$ is the
length of the longest decreasing subsequence of $(w(1), \dots, w(n))$
and $\rev{w} = (w(n),\dots, w(1))$ is the reverse of $w$, already
considered in the proof of Proposition \ref{p4}.  The map
\[
\textstyle \sum a_w w \mapsto \sum a_w \rev{w} = \sum a_w ww_0
\quad (a_w \in \Bbbk)
\]
given by right multiplication by $w_0$ defines a linear involution of
$\Bbbk[W_n]$ carrying $\{w \mid \DS(w) \ge k\}$ onto $\{w \mid
\text{IS}(w) \ge k\}$, for any $k$. It induces a linear involution on
$\im(\Phi)$ which is given by right multiplication by the matrix
$P(w_0)^{\otimes r}$; this clearly interchanges the two bases.
\end{proof}

\begin{rmk}\label{r4}
Theorem \ref{t1} has a counterpart for the $\Bbbk$-submodule
$\V^{\otimes r} \otimes \vv_n$ of $\V^{\otimes (r+1)}$, which we
identify with $\V^{\otimes r}$.  The restriction of the diagonal
action of $W_n$ to the subgroup
\[
W_{n-1} = \{ w \in W_n \mid w(n)=n \} \subset W_n
\]
gives an action on $\V^{\otimes r}$ fixing $\vv_n$. Let $\Phi':
\Bbbk[W_{n-1}] \to \End_\Bbbk(\V^{\otimes r})$ be the corresponding
representation.  Then \emph{the set of $P(w)^{\otimes r}$ indexed by
  $w \in W_{n-1}$ having an increasing (resp., decreasing) subsequence
  of length $n-1-r$ is a basis of $\im(\Phi')$}. The proof is nearly
identical with that of Theorem \ref{t1}; we leave the details to the
reader.
\end{rmk}

\section{Connections with Schur--Weyl duality}\label{s4}\noindent
Let $\A = \Bbbk[W_n]$ be the group algebra of $W_n$. The diagonal
action of $W_n$ makes $\V^{\otimes r}$ into an $\A$-module. Let
\[
\A' = \End_\A(\V^{\otimes r}) = \{\varphi \in \End_\Bbbk(\V^{\otimes r})
\mid \varphi(\alpha t) = \alpha \varphi(t), \alpha \in \A, t \in
\V^{\otimes r} \},
\]
the commutant of $\A$.  Then $\V^{\otimes r}$ is also an $\A'$-module,
with $\varphi \in \A'$ acting by $\varphi \cdot t = \varphi(t)$ for
any $t \in \V^{\otimes r}$. Each $\alpha \in \A$ induces an
$\A'$-homomorphism $f_\alpha: \V^{\otimes r} \to \V^{\otimes r}$
defined by $f_\alpha(t) = \alpha t$. Now consider the bicommutant
(double centralizer)
\[
\A'' = \End_{\A'}(\V^{\otimes r}) = \{\psi \in \End_\Bbbk(\V^{\otimes
  r}) \mid \psi f = f \psi, \text{ for all } f \in \A' \}
\]
where the multiplication here is functional composition.  Then the map
\begin{equation}\label{eq:dc} 
\Phi: \A \to \End_{\A'}(\V^{\otimes r}) = \A'', \quad \alpha \mapsto f_\alpha. 
\end{equation}
is an $\Bbbk$-algebra homomorphism. It is abstractly the same map as
the representation $\Phi$ considered in \eqref{e:Phi}, with restricted
codomain.

If $\Bbbk$ is a field of characteristic zero or characteristic larger
than $n$, then $\V^{\otimes r}$ is semisimple as an $\A$-module and
Jacobson's density theorem \cite{Jacobson:45} (see also
\cite{Jacobson}*{\S4.3} or \cite{Lang}*{Chap.~XVII, Theorem 3.2})
implies that $\Phi$ is surjective.  By the main result of \cite{BDM}
(see also \cite{Donkin}*{\S6}), $\Phi$ is surjective in general, for
any unital commutative ring $\Bbbk$.

Let $\Ptn(r,n)$ be the partition algebra \cites{Martin:book,
  Martin:94, Martin:96, Jones} over $\Bbbk$ on $2r$ vertices with
parameter $n$. It has a basis indexed by the set partitions
(equivalence relations) on the set $\{1,\dots,r, 1', \dots, r'\}$;
basis elements are often depicted by diagrams on $2r$ vertices labeled
by elements of that set, with a path connecting two vertices if and
only if they lie in the same subset of the set partition. The action
of $\Ptn(r,n)$ on $\V^{\otimes r}$ is described explicitly in
\cite{HR}, to which we refer for basic properties of partition
algebras. Let $\Psi: \Ptn(r,n) \to \End_\Bbbk(\V^{\otimes r})$ be the
representation afforded by the action.

\begin{prop}[Schur--Weyl duality]\label{p6}
The commutant $\A' = \End_{W_n}(\V^{\otimes r})$ is the image of the
representation $\Psi: \Ptn(r,n) \to \End_\Bbbk(\V^{\otimes r})$. The
bicommutant $\A'' = \im(\Phi)$ is the centralizer algebra
$\End_{\Ptn(r,n)}(\V^{\otimes r})$.
\end{prop}

\begin{proof}
The fact that $\A' = \End_{W_n}(\V^{\otimes r})$ is the image of
$\Psi$ is \cite{HR}*{Thm.~3.6}; the combinatorial proof given there is
valid over any $\Bbbk$. The aforementioned surjectivity of $\Phi$ then
implies the second claim.
\end{proof}

If we now assume that $\Bbbk=\R$, the following shows that the study
of the set of nonnegative invariants in $\End_{\Ptn(r,n)}(\V^{\otimes
  r})$ reduces to the study of the set $\Omega$ of doubly stochastic
elements of $\im(\Phi)$.

\begin{cor*}
Assume that $\Bbbk=\R$. Then the set of all nonnegative matrices in
$\im(\Phi) = \End_{\Ptn(r,n)}(\V^{\otimes r})$ identifies with the set
of nonnegative scalar multiples of the set $\Omega$ of doubly
stochastic matrices in $\im(\Phi)$.
\end{cor*}

\begin{proof}
If $M$ is a matrix in $\im(\Phi)$, it may be written as a linear
combination of elements of the set $\Gamma$ of $P(w)^{\otimes r}$, for
$w \in W_n$. The rows and columns of the $P(w)^{\otimes r}$ all sum to
$1$; hence the rows and columns of $M$ all sum to the same value.  If
the entries of $M$ are nonnegative, then so is the common value $s$ of
the row and column sums. If $s\ne 0$ then $s^{-1}M$ is doubly
stochastic, so $M$ is a positive multiple of that doubly stochastic
matrix. If $s=0$ then $M = [0]$ must be the zero matrix, which is also
zero times a doubly stochastic matrix.

Conversely, suppose that $D \in \im(\Phi)$ is doubly stochastic. Then
it is a nonnegative matrix in $\End_{\Ptn(r,n)}(\V^{\otimes r})$,
hence the same is true of any nonnegative scalar multiple.
\end{proof}

\begin{rmk}\label{r5}
Return to general $\Bbbk$. In the situation of Remark \ref{r4}, there
is an action of the ``half'' partition algebra $\Ptn(r+\frac{1}{2},
n)$ on $\V^{\otimes r} \cong \V^{\otimes r} \otimes \vv_n$, where
$\Ptn(r+\frac{1}{2}, n)$ is the subalgebra of $\Ptn(r+1, n)$ spanned
by all diagrams with an edge connecting vertices $r$, $r'$; see
\cite{HR} for details. All of the results in this section generalize
to the half partition algebra. In particular, the bicommutant of the
action of $W_{n-1}$ is equal to
\[
\im(\Phi') = \End_{\Ptn(r+\frac{1}{2}, n)}(\V^{\otimes r}).
\]
Again, we leave the details to the interested reader. Remark \ref{r4}
gives a basis of this algebra.
\end{rmk}

\section{Equations for $\im(\Phi)$ and $\Omega$}\label{s5}\noindent
There is another symmetric group $\Sym_r$ acting on $\V^{\otimes r}$,
by place-permutation, and its commutant algebra is the Schur algebra
$\End_{\Sym_r}(\V^{\otimes r})$ studied in \cites{Green,S-Martin},
etc. We write it as $\Sym_r$ to emphasize that the actions of $\Sym_r$
and $W_n$ on tensors are very different.  Write $(i_1\cdots
i_r)^\sigma$ for the effect of place-permuting $i_1 \cdots i_r$
according to $\sigma \in \Sym_r$.  By \cite{BDM}*{Prop.~3.2}, combined
with the fact that $\Ptn(r,n)$ is generated by $\Sym_r$ and the
elements $p_1$, $p_{3/2}$ in the notation of \cite{HR}*{(1.10)}, an
$n^r \times n^r$ matrix $X = [x_{i_1\cdots i_r,\, j_1\cdots j_r}]$
belongs to the bicommutant $\A'' = \im(\Phi)$ if and only if
\begin{enumerate}\renewcommand{\labelenumi}{(\roman{enumi})}
\item $x_{i_1\cdots i_r,\, j_1\cdots j_r} = x_{(i_1\cdots i_r)^\sigma,
  (j_1\cdots j_r)^\sigma}$, for all $\sigma \in \Sym_r$.
\item $x_{i_1\cdots i_r,\, j_1\cdots j_r} = 0$ if ($i_1 = i_2$ but $j_1 \ne j_2$) or
  ($i_1 \ne i_2$ but $j_1 = j_2$).
\item $\sum_{i=1}^n x_{i\, i_2\cdots i_r,\, j_1\cdots j_r} =
  \sum_{j=1}^n x_{i_1\cdots i_r,\, j\, j_2\cdots j_r}$, for all
  $i_1,\dots, i_r,\, j_1,\dots, j_r$.
\end{enumerate}
Condition (i) is the condition that $X$ is in the Schur algebra, and
(iii) is equivalent to $X$ commuting with $J_n \otimes {I_n}^{\otimes
  (r-1)}$, where $J_n=[1]$ is the $n \times n$ matrix of all $1$'s and
$I_n=[\delta_{ij}]$ is the $n \times n$ identity matrix. Thanks to
(i), conditions (ii), (iii) can be place-permuted to any other places.

Finally, if $\Bbbk = \R$, including the additional conditions
\begin{enumerate}\renewcommand{\labelenumi}{(\roman{enumi})}
  \setcounter{enumi}{3}
\item $x_{i_1\cdots i_r,\, j_1\cdots j_r} \ge 0$,
\item $\sum_{i_1\cdots i_r} x_{i_1\cdots i_r,\, j_1\cdots j_r} = 1 =
  \sum_{j_1\cdots j_r} x_{i_1\cdots i_r,\, j_1\cdots j_r}$
\end{enumerate}
(for all $i_1\cdots i_r,\, j_1\cdots j_r$) along with conditions
(i)--(iii) gives a description of the set $\Omega$ of doubly
stochastic elements of $\im(\Phi) = \End_{P_r(n)}(\V^{\otimes r})$.

\subsection*{Acknowledgments}\noindent
The author is grateful to R.A.~Brualdi and David Roberson for useful
discussions on an earlier version of this manuscript.

%%%%%%%%%%%%%%%%%%%%%%%%%%%%%%%%%%%%%%%%%%%%%%%%%%%%%%%%%%%%%%%%%%%%%%%%%%%%

\end{document}